%% file: talbot_head.tex
\documentclass[12pt]{amsart}
\usepackage{geometry}                
\usepackage{graphicx}
\usepackage{url}
\usepackage{amsmath}
\usepackage{amsthm}
\usepackage{amssymb}
\usepackage{epstopdf}
\usepackage{fancyhdr}

\usepackage{tikz} 
\usetikzlibrary{chains,shapes,arrows,trees,matrix,positioning,backgrounds,fit,calc,fadings}

\def\op#1{\text{#1}}
\def\ring#1{{\Bbb #1}}
\def\VR{V_{\ring{R}}}

\title{Walter Talbot's thesis}
\author{Thomas Hales}
\date{May 16, 2018}                                           

\newtheorem{lemma}{Lemma}

\begin{document}
\maketitle

\lhead{Talbot's thesis}
\rhead{}

\parskip=0.8\baselineskip

\input{talbot}
\bibliographystyle{plain}      

\bibliography{ref} 

\end{document}

%% file: talbot.tex
Walter Richard Talbot was the fourth African American to earn a PhD in
Mathematics.  His doctoral degree is from the University of Pittsburgh
in 1934 in geometric group theory.  As far as I know, research in
Pittsburgh on geometric group theory seems to have started with
Professors J.  S. Taylor and M. M. Culver at the University of
Pittsburgh, who were both academic descendents of Felix Klein.  They
were both members of Talbot's thesis committee.

Talbot's research grew out of Klein's work on fundamental domains
\cite{MR0247997} \cite{MR0247996}.  A contemporary research program
was the determination of fundamental domains of finite group actions
on complex vector spaces
\cite{MR1506344},\cite{MR1559065},\cite{phdthesis-M}.
His thesis is not widely available, and this note gives a brief
synopsis of the main results of his thesis, expressed using modern
mathematical methods and language, and placed in general context
\cite{MR573930}, \cite{phdthesis-T}.

\section{A fundamental domain for the symmetric group}

Let $W$ be a finite group acting on a a finite-dimensional complex
vector space $V$.
The general research problem is to give an explicit description
of a fundamental domain.

Let $\op{Herm}(V)$ be the space of hermitian forms on $V$; that
is,  the real vector space of space sesquilinear forms
$h:V\times V\to\ring{C}$ such that
\[
\overline {h(u,v)} = h(v,u),\quad h( u,\mu v) =
 \mu\, h(u,v), \quad u,v\in V,\quad
\mu\in \ring{C}. 
\]
The general linear group $GL(V)$ of $V$ acts on
$\op{Herm}(V)$ by
\[
(g\cdot h)(u,v) = h(g^{-1} u, g^{-1} v),\quad g\in GL(V)
\]
and restricts to a representation of $W$ on $\op{Herm}(V)$.

Consider the following special context.
Let $W = S_n$, the symmetric group on $n$ letters.  Take
$V\subset\ring{C}^n$ to be the 
standard (irreducible) $n-1$-dimensional representation of $W$ on
\[
V=\{(z_1,\ldots,z_n) \mid z_1 + \cdots + z_n = 0\}
\]
occuring in the permutation representation of $W$ on $\ring{C}^n$.
Each irreducible representation of $S_n$ can be realized over
$\ring{R}$.  Pick a real form $\VR$ of $V$.  The group $W=S_n$ acts on
$\op{Herm}(V)$ as above.  Its decomposition into irreducibles is given
as follows, with conventional notation:
\begin{lemma}
\[
\op{Herm}(V) = {1} \oplus \VR \oplus V_{n-2,2} \oplus V_{n-2,1,1}.
\]
\end{lemma}

\begin{proof}
By breaking $\op{Herm}(V)$ into real and imaginary parts, the
representation splits as
  \[
  \op{Sym}^2(\VR) \oplus i \Lambda^2(\VR) \cong
  \op{Sym}^2(\VR) \oplus
  \Lambda^2(\VR) \cong ({1} \oplus \VR \oplus V_{n-2,2}) \oplus (V_{n-2,1,1})
  \cong \VR\otimes \VR,
  \]
as given in Exercise 4.19 in \cite{fulton}. 
\end{proof}

If $h$ is a hermitian form, then  its zero set
\[
Z(h) = \{u\in V\mid h(u,u) = 0\}
\]
partitions $V$ into positive $h \ge 0$ and negative $h \le 0$ regions.
Given hermitian forms $h_i$, we write $h_i \ge h_j$ to indicate the
positive region of $V$ bounded by the zero-set $Z(h_i-h_j)$.

Select Hermitian forms
\[
h_1, h_2,\ldots, h_n
\]
giving the permutation basis of $1\oplus \VR \subset \op{Herm}(V)$.
Define chambers $T(g)$ in $V$, for $g\in S_n$, by inequalities
\begin{equation}\label{eqn:f}
h_{g1} \ge \cdots \ge h_{gn}.
\end{equation}
These chambers partition $V$ and have walls given by hypersurfaces
$Z(h_i-h_j)$, where $h_i - h_j$ span the standard representation $\VR$
of $S_n$, realized in the space of Hermitian matrices.  Explicitly,
the representation $1\oplus \VR$ has permutation basis
\[
h_i(z,z') = \bar{z}_i z'_i,\quad z,z'\in V \subset \ring{C}^n,\quad i=1,\ldots,n.
\]
Thus,  Talbot's chambers (\ref{eqn:f}) of $V$ are given by
\[
|z_{g1}|^2 \ge \cdots \ge |z_{gn}|^2.
\]
There is an obvious correspondence with the Weyl chambers of the action of
$S_n$ on $\VR$.  (For a point of historical comparison, Weyl gave
lectures on the structure and representation of continuous groups in
Princeton in 1933-34.)  More to the point, as B. Ion has observed,
the proper setting for Talbot's thesis is the Cartan decomposition
for $GL(n,\ring{C})/U(n)$, where the space of Hermitian matrices
is the Lie algebra complement of the Lie algebra of $U(n)$,
and $1\oplus V_\ring{R}\subset \op{Herm}(V)$ is the maximal abelian with
an action of the Weyl group $W=S_n$.

\section{End notes}

Talbot works projectively in $\ring{P}V$ rather than $V$, but this
does not play any role in the results.

Talbot only works with $n=5$ and in fact with the alternating group
$A_5$.  When $n=5$, there is another subspace $1\oplus
V_{3,2}\subset \op{Herm}(V)$ that is (up to the sign character) the
permutation representation $1\oplus V_{2,2,1}$ of $S_5$ on six letters
(equivalent to the the action of $S_5$ on its six Sylow-$5$
subgroups).  Talbot also considers regions of $V$ cut out by a
permutation basis of hermitian forms in this representation.  He also
analyzes the boundary of chambers.
